\documentclass[smallextended,referee,envcountsect,]{svjour3}
\smartqed
\usepackage{graphicx}
\journalname{Optimization Letters}

\begin{document}

\title{Global Convergence of Algorithms Based on Unions of Nonexpansive Maps}


\author{Alexander J. Zaslavski}

\institute{Alexander J. Zaslavski \at
             Department of Mathematics \\
              The Technion -- Israel Institute of Technology \\
              Haifa, 32000, Israel\\
              ajzasl@technion.ac.il
              }

\date{Received: date / Accepted: date}

\maketitle

\begin{abstract}
In his recent research  M. K. Tam (2018) considered a framework for the analysis of iterative algorithms which can be described in terms of a structured set-valued operator.
At each point in the ambient space, the value of the operator
can be expressed as a finite union of values of single-valued paracontracting operators.
He showed that the associated fixed point iteration is locally convergent around strong fixed points. This result generalizes a theorem due to Bauschke and Noll (2014).
In the present paper we generalize the result of Tam and show the global convergence of his algorithm for an arbitrary starting point. An analogous result is also proved for the Krasnosel'ski-Mann iterations.
\end{abstract}


\section{Introduction}

For more than sixty years now, there has been a lot of research activity
in the study of the fixed point theory of  nonexpansive operators \cite{alwa,g17,goeb90a,gr84a,khak15,kop12,rz14a,qin,z16a,z18a}.
The starting point of this study is  Banach's celebrated
theorem \cite{ban22a} concerning the existence of a unique fixed point for a strict contraction.
It also concerns the convergence of (inexact) iterates of a nonexpansive mapping to one of its fixed points.
Since that seminal result, many developments have taken place in this field including, in particular, studies of feasibility, common fixed point
problems and variational inequalities, which find important applications in mathematical analysis, optimization theory, and in engineering, medical
and the natural sciences \cite{baubor96a,cen18a,g17,gr17,qin,txy,z16a,z18a}. In particular in \cite{tam18a}, it was considered a framework for the analysis of iterative algorithms, which can be described in terms of a structured set-valued operator. Namely, at every point in the ambient  space, it is assumed that the value of the operator can be expressed as a finite union of values of single-valued paracontracting operators.
For such algorithms it was shown in \cite{tam18a}  that the associated fixed point iteration is locally convergent around strong fixed points.
In \cite{dao} an analogous result was obtained for Krasnosel'ski-Mann iterations.
The  result of \cite{tam18a} generalizes a theorem due to Bauschke and Noll \cite{baus14a}.
In the present paper we generalize the main result of \cite{tam18a} and show the global convergence of the  algorithm for an arbitrary starting point. An analogous result is also proved for the Krasnosel'ski-Mann iterations.

\section{Global convergence of iterates}

Suppose  that $(X,\rho)$ is a metric  space  and that $C \subset X$ is its nonempty, closed set. For every point  $x \in X$ and every positive number $r$ define
$$B(x,r)=\{y \in X:\; \rho(x,y)  \le r\}.$$
For each $x \in X$ and each nonempty set $D\subset X$ set
$$\rho(x,D)=\inf\{\rho(x,y):\;y \in D\}.$$
For every operator  $S:C \to C$ set
$$\mbox{Fix}(S)=\{x \in C:\; S(x)=x\}.$$
Fix
$$\theta \in C.$$
Suppose that the following assumption holds:

(A1) For each  $M>0$ the set $B(\theta,M)\cap C$ is compact.

Assume that  $m$ is a natural number, $T_i:C \to C$, $i=1,\dots,m$ are continuous operators
and that the following assumption holds:

(A2) For every natural number  $i \in \{1,\dots,m\}$, every point  $z \in \mbox{Fix}(T_i)$,  every point  $x \in C$ and every $y \in C\setminus \mbox {Fix}(T_i)$, we have
$$\rho(z,T_i(x)) \le \rho(z,x)$$
and
$$\rho(z,T_i(y))  < \rho(z,y).$$

Note that operators satisfying (A2) are called paracontractions \cite{els}.

Assume that for every point $x \in X$, a nonempty set
$$\phi(x)\subset \{1,\dots,m\} \eqno (1)$$
is given. In other words, $$\phi:X\to 2^{\{1,\dots,m\}}\setminus \{\emptyset\}.$$
Suppose that the following assumption holds:

(A3) For each $x \in C$ there exists $\delta>0$ such that for each $y \in B(x,\delta)\cap C$,
$$\phi(y) \subset \phi(x).$$

Define
$$T(x)=\{T_i(x):\; i \in \phi(x)\} \eqno (2)$$
for each $x \in C$,
$$\bar F(T)=\{z \in C:\;T_i(z)=z,\; i=1,\dots,m\} \eqno (3)$$
and
$$F(T)=\{z \in C:\; z \in T(z)\}. \eqno (4)$$

Assume that
$$\bar F(T)\not =\emptyset.$$
Denote by Card$(D)$ the cardinality of a set $D$. For each $z \in R^1$ set
$$\lfloor z\rfloor =\inf\{i:\; i \mbox{ is an integer and } i \le z\}.$$
In the sequel we suppose that the sum over empty set is zero.

We study the asymptotic behavior of sequences of iterates $x_{t+1} \in F(x_t)$, $t=0,1,\dots$. In particular we are interested in their convergence to a fixed point of $T$. 
This iterative algorithm was introduced in \cite{tam18a} which also contains its  application to sparsity constrained minimisation.

The following result, which is proved in Section 4,
shows that almost all iterates of our set-valued mappings are approximated solutions of the corresponding fixed point problem.
Many results of  this type are collected in \cite{z16a,z18a}.

\begin{theorem}
Assume that $M>0$, $\epsilon \in (0,1)$ and that
$$\bar F(T)\cap B(\theta,M) \not =\emptyset. \eqno (5)$$
Then there exists an integer $Q\ge 1$ such that for each sequence $\{x_i\}_{i=0}^{\infty}\subset C$
 which satisfy
$$\rho(x_0,\theta) \le M $$
and 
$$x_{t+1}\in F(x_t) \mbox { for each integer } t \ge 0$$
the inequality
$$\rho(x_t,\theta)\le 3M$$ holds for all integers $t \ge 0$,
$$\mbox{Card}(\{t \in \{0,1,\dots,\}:\;\rho(x_t,x_{t+1})> \epsilon\})\le Q$$
and $\lim_{t \to \infty}\rho(x_t,x_{t+1})=0$.\end{theorem}

The following global convergence result is proved in Section 5.

\begin{theorem}
Assume that  a  sequence $\{x_t\}_{t=0}^{\infty}\subset C$
and  that for each integer $t \ge 0$,
$$x_{t+1}\in F(x_t).$$
Then there exist
$$x_*=\lim_{t \to \infty}x_t$$ and a natural number $t_0$ such that for each integer $t \ge t_0$
$$\phi(x_t)\subset \phi(x_*)$$
and if an integer $i \in \phi(x_t)$ satisfies $x_{t+1}=T_i(x_t)$, then
$$T_i(x_*)=x_*.$$
\end{theorem}

Theorem 2.2 generalizes the main result of \cite{tam18a} which establishes a local convergence of the iterative algorithm for iterates starting from a point belonging to a neighborhood of a strong fixed point belonging to the set $\bar F(T)$.

\section{An auxiliary result}

\begin{lemma}
Assume that $M,\epsilon>0$ and that $z_* \in C$ satisfies
$$T_i(z_*)=z_*,\;i=1,\dots,m. \eqno (6)$$
Then there exists $\delta>0$ such that for each $s \in \{1,\dots,m\}$ and each $x \in C\cap B(\theta,M)$ satisfying
$$\rho(x,T_s(x))>\epsilon \eqno (7)$$
the inequality
$$\rho(z_*,T_s(x))\le \rho(z_*,x)-\delta \eqno (8)$$
is true.
\end{lemma}

\begin{proof}
Let $s \in \{1,\dots,m\}$. It is sufficient to show that there exists $\delta>0$ such that for each $x \in C\cap B(\theta,M)$ satisfying (7) inequality (8) is true. Assume the contrary. Then for each integer $k \ge 1$, there exists
$$x_k \in C\cap B(\theta,M) \eqno (9)$$
such that
$$\rho(x_k,T_s(x_k))>\epsilon \eqno (10)$$
and
$$\rho(z_*,T_s(x_k))> \rho(z_*,x_k)-k^{-1}. \eqno (11)$$
In view of (A1) and (9), extracting a subsequence and re-indexing, we may assume without loss of generality that there exists
$$x_*=\lim_{k \to \infty}x_k. \eqno (12)$$
By (9)-(12) and the continuity of $T_s$,
$$\rho(x_*,\theta)\le M,$$
$$\rho(x_*,T_s(x_*))=\lim_{k \to \infty}\rho(x_k,T_s(x_k))\ge  \epsilon$$
and
$$\rho(z_*,T_s(x_*))\ge \rho(z_*,x_*).$$
This contradicts (6) and (A2). The contradiction we have reached proves Lemma 3.1.
\end{proof}

\section{Proof of Theorem 2.1} By (5), there exists
$$z_* \in B(\theta,M)\cap \bar F(T). \eqno (13)$$
Lemma 3.1 implies that there exists $\delta \in (0,\epsilon)$ such that the following property holds:

(a) for each $s \in \{1,\dots,m\}$ and each $x \in C\cap B(z_*,2M)$ satisfying
$$\rho(x,T_s(x))>\epsilon $$
we have
$$\rho(z_*,T_s(x))\le \rho(z_*,x)-\delta.$$

Choose a natural number
$$Q \ge 2M\delta^{-1}. \eqno (14)$$
Assume that $\{x_i\}_{i=0}^{\infty}\subset C$,
$$\rho(x_0,\theta) \le M \eqno (15)$$
and that for each integer $t \ge 0$,
$$x_{t+1}\in F(x_t). \eqno (16)$$
Let $t \ge 0$ be an integer. By (2) and (16), there exists $s \in \{1,\dots,m\}$ such that
$$x_{t+1}=T_s(x_t). \eqno (17)$$
Assumption (A2) and equations (3), (13) and (17) imply that
$$\rho(z_*,x_{t+1}) =\rho(z_*,T_s(x_t))\le \rho(z_*,x_t). \eqno (18)$$
Since $t$ is an arbitrary nonnegative integer equations (13), (15) and (18) imply that for each integer $i \ge 0$,
$$\rho(z_*,x_i) \le \rho(z_*,x_0) \le 2M \eqno (19)$$
and
$$\rho(x_i,\theta)\le 3M.$$
Assume that
$$\rho(x_{t+1},x_t)>\epsilon. \eqno (20)$$ Property (a) and equations (17), (19) and (20) imply that
$$\rho(z_*,x_{t+1})=\rho(z_*,T_s(x_t))\le \rho(z_*,x_t)-\delta.$$
Thus we have shown that the following property holds:

(b) if an integer $t \ge 0$ satisfies (20), then
$$\rho(z_*,x_{t+1})\le \rho(z_*,x_t)-\delta.$$

Assume that $n \ge 1$ is an integer. Property (b) and equations (18)-(20) imply that
$$2M \ge \rho(z_*,x_0) \ge \rho(z_*,x_0)-\rho(z_*,x_{n+1})$$
$$=\sum_{t=0}^n(\rho(z_*,x_t)-\rho(z_*,x_{t+1}))$$
$$\ge \sum\{\rho(z_*,x_t)-\rho(z_*,x_{t+1}):\; t \in \{0,\dots,n\},\; \rho(x_t,x_{t+1})>\epsilon\}$$
$$\ge \delta \mbox{Card} (\{t \in \{0,\dots,n\}:\; \rho(x_t,x_{t+1})>\epsilon\})$$
and in view of (14),
$$\mbox{Card} (\{t \in \{0,\dots,n\}:\; \rho(x_t,x_{t+1})>\epsilon\})\le 2M\delta^{-1}\le  Q.$$
Since $n$ is an arbitrary natural number we conclude that
$$\mbox{Card} (\{t \in \{0,1,\dots\}:\; \rho(x_t,x_{t+1})>\epsilon\})\le   Q.$$
Since $\epsilon$ is any element of $(0,1)$ Theorem 2.1 is proved.

\section{Proof  of Theorem 2.2} In view of Theorem 2.1, the sequence $\{x_t\}_{t=0}^{\infty}$ is bounded. In view of (A1), it has a limit point $x_* \in C$
and a subsequence $\{x_{t_k}\}_{k=0}^{\infty}$ such that
$$x_*=\lim_{k \to \infty}x_{t_k}. \eqno (21)$$ In view of (A3) and (21), we may assume without loss of generality that
$$\phi(x_{t_k}) \subset \phi(x_*),\;k=1,2,\dots \eqno (22)$$
and that there exists
$$\widehat p \in \phi(x_*)$$
such that
$$x_{t_k+1}=T_{\widehat p}(x_{t_k}),\;k=1,2,\dots. \eqno (23)$$
It follows from Theorem 2.1, the continuity of $T_{\widehat p}$ and equations (21) and (23) that
$$T_{\widehat p}(x_*)=\lim_{k \to \infty}T_{\widehat p}(x_{t_k})=\lim_{k \to \infty}x_{t_k+1}=\lim_{k \to \infty} x_{t_k}=x_*. \eqno (24)$$
Set
$$I_1=\{i \in \phi(x_*):\; T_i(x_*)=x_*\},\; I_2=\phi(x_*)\setminus I_1. \eqno (25)$$
In view of (24) and (25),
$$\widehat p \in I_1.$$
Fix $\delta_0\in (0,1)$ such that
$$\rho(x_*,T_i(x_*))>2\delta_0,\; i \in I_2. \eqno (26)$$
Assumption (A3), the continuity of $T_i,\;i=1,\dots,m$ and (26) imply that there exists $\delta_1 \in (0,\delta_0)$ such that for each $x \in B(x_*,\delta_1)\cap C$,
$$\phi(x)\subset \phi(x_*), \eqno (27)$$
$$\rho(x,T_i(x))>\delta_0,\; i \in I_2. \eqno (28)$$
Theorem 2.1 implies that there  exists an integer $q_1\ge 1$ such that for each integer $t \ge q_1$,
$$\rho(x_t,x_{t+1}) \le \delta_0/2. \eqno (29)$$
Assume that
$$\epsilon \in (0,\delta_1), \eqno (30)$$
$$t \ge q_1 \eqno (31)$$
is an integer and that
$$\rho(x_t,x_*) \le \epsilon. \eqno (32)$$
It follows from (27), (28), (30) and (32) that
$$\phi(x_t) \subset \phi(x_*) \eqno (33)$$
and
$$\rho(x_t,T_i(x_t))>\delta_0,\; i \in I_2. \eqno (34)$$
In view of (33), there exists
$$s \in \phi(x_*)$$ such that
$$x_{t+1}=T_s(x_t). \eqno (35)$$. By (29), (31) and (35),
$$\rho(x_t,T_s(x_t))=\rho(x_t,x_{t+1}) \le \delta_0/2. \eqno (36)$$
It follows from (25), (34) and (36) that
$$s \in I_1,\; T_s(x_*)=x_*.$$
Combined with assumption (A2) and equations (32) and (35) this implies that
$$\rho(x_{t+1},x_*) =\rho(T_s(x_t),x_*)\le \rho(x_t,x_*) \le \epsilon.$$
Thus we have shown that if $t \ge q_1$ is an integer and (32) holds, then (33) is true and if $s \in \phi(x_*)$ and (35) holds, then $s \in I_1$ and $\rho(x_{t+1},x_*)\le \epsilon$.
By induction and (21), we obtain that
$$\rho(x_i,x_*) \le \epsilon $$
for all sufficiently large natural numbers $i$. 
Since $\epsilon $ i an arbitrary element of $(0,\delta_1)$ we conclude that
$$\lim_{t \to \infty}x_t=x_*$$ and Theorem 2.2 is proved.

\section{Kraasnosel'ski-Mann iterations}

Assume that $(X,\|\cdot\|)$ is a normed space and that $\rho(x,y)=\|x-y\|,\;x,y \in X$.
We use the notation, definitions and assumptions introduced in Section 2. In particular, we assume that assumptions (A1)-(A3) hold. Suppose that the set $C$ is convex and denote by $Id:X\to X$ the identity operator: $Id(x)=x$, $x \in X$.
Let
$$\kappa \in (0,2^{-1}).$$
We consider Kraasnosel'ski-Mann iteration associated with our set-valued mapping $T$ and obtain the global convergence result (see  Theorem 6.2 below) which generalizes  
the local convergence result of \cite{dao}  for iterates starting from a point belonging to a neighborhood of a strong fixed point belonging to the set $\bar F(T)$.

The following result is proved in Section 7.

\begin{theorem}
Assume that $M>0$, $\epsilon \in (0,1)$ and that
$$\bar F(T)\cap B(\theta,M) \not =\emptyset \eqno (37)$$
Then there exists an integer $Q\ge 1$ such that for each
$$\{\lambda_t\}_{t=0}^{\infty}\subset (\kappa,1-\kappa) \eqno (38)$$
and each
sequence $\{x_i\}_{i=0}^{\infty}\subset C$
 which satisfies
$$\|x_0-\theta\| \le M $$
and 
$$x_{t+1}\in (1-\lambda_t)x_t+\lambda_tT(x_t) \mbox { for each integer }t\ge 0 \eqno (39)$$
the inequality
$$\|x_t-\theta\|\le 3M$$ holds for all integers $t \ge 0$,
$$\mbox{Card}(\{t \in \{0,1,\dots,\}:\;\|x_t-x_{t+1}\|> \epsilon\})\le Q$$
and $\lim_{t \to \infty}\|x_t-x_{t+1}\|=0$.\end{theorem}

The following result is proved in Section 8.

\begin{theorem}
Assume that
$$\{\lambda_t\}_{t=0}^{\infty}\subset (\kappa,1-\kappa) $$
and that
a  sequence $\{x_t\}_{t=0}^{\infty}\subset C$
satisfies (39).
Then there exist
$$x_*=\lim_{t \to \infty}x_t$$ and a natural number $t_0$ such that for each integer $t \ge t_0$
$$\phi(x_t)\subset \phi(x_*)$$
and if an integer $i \in \phi(x_t)$ satisfies $$x_{t+1}=\lambda_tT_i(x_t)+(1-\lambda)x_t,$$ then
$$T_i(x_*)=x_*.$$
\end{theorem}

\section{Proof of Theorem 6.1}

By (37), there exists
$$z_* \in B(\theta,M)\cap \bar F(T). \eqno (40)$$
Lemma 3.1 implies that there exists $\delta \in (0,\epsilon)$ such that the following property holds:

(c) for each $s \in \{1,\dots,m\}$ and each $x \in C\cap B(z_*,2M)$ satisfying
$$\rho(x,T_s(x))>\epsilon $$
we have
$$\rho(z_*,T_s(x))\le \rho(z_*,x)-\delta.$$

Choose a natural number
$$Q \ge 2M\delta^{-1}\kappa^{-1}. \eqno (41)$$
Assume that (38) holds and that a sequence $\{x_i\}_{i=0}^{\infty}\subset C$ satisfies (39) and
$$\|x_0-\theta\| \le M. \eqno (42)$$
Let $t \ge 0$ be an integer. By (2) and (39), there exists $s \in \{1,\dots,m\}$ such that
$$x_{t+1}=\lambda_tT_s(x_t)+(1-\lambda_t)x_t. \eqno (43)$$
Assumption (A2) and equations (3), (40) and (43) imply that
$$\|x_{t+1}-z_*\| =\|\lambda_tT_s(x_t)+(1-\lambda_t)x_t-z_*\|$$
$$\le \lambda_t\|T_s(x_t)-z_*\|+(1-\lambda_t)\|x_t-z_*\|\le \|z_*-x_t\|. \eqno (44)$$
Since $t$ is an arbitrary nonnegative integer equations (40), (42) and (44) imply that for each integer $i \ge 0$,
$$\|z_*-x_i\| \le \|z_*-x_0\| \le 2M $$
and
$$\|x_i-\theta\|\le 3M.$$
Assume that
$$\|x_{t+1}-x_t\|>\epsilon. \eqno (45)$$
It follows from (38),  (43) and (45) that
$$\epsilon<\|x_{t+1}-x_t\|=\|\lambda_tT_s(x_t)+(1-\lambda_t)x_t-x_t\|=\lambda_t\|T_s(x_t)-x_t\|$$
and
$$\|T_s(x_t)-x_t\|\ge \epsilon \lambda_t^{-1}\ge \epsilon(1-\kappa)^{-1}. \eqno (46)$$
Property (c) and equation (46) imply that
$$\|z_*-T_s(x_t)\|\le \|z_*-x_t\|-\delta. \eqno (47)$$
By (38), (43) and (47),
$$\|x_{t+1}-z_*\|=\|\lambda_tT_s(x_t)+(1-\lambda_t)x_t-z_*\|$$
$$\le \lambda_t\|T_s(x_t)-z_*\|+(1-\lambda_t)\|x_t-z_*\|$$
$$\le  \lambda_t(\|x_t-z_*\|-\delta)+(1-\lambda_t)\|x_t-z_*\|$$
$$\le\|x_t-z_*\|-\lambda_t\delta \le \|x_t-z_*\|-\delta\kappa. \eqno (48)$$

Thus we have shown that the following property holds:

(d) if an integer $t \ge 0$ satisfies (45), then
$$\|z_*-x_{t+1}\|\le \|z_*-x_t\|-\delta\kappa.$$

Assume that $n \ge 1$ is an integer. Property (d) and equations (40), (42) and (44) imply that
$$2M \ge \|z_*-x_0\| \ge \|z_*-x_0\|-\|z_*-x_{n+1}\|$$
$$=\sum_{t=0}^n(\|z_*-x_t\|-\|z_*-x_{t+1}\|)$$
$$\ge \sum\{\|z_*-x_t\|-\|z_*-x_{t+1}\|:\; t \in \{0,\dots,n\},\; \|x_t-x_{t+1}\|>\epsilon\}$$
$$\ge \delta \kappa \mbox{Card} (\{t \in \{0,\dots,n\}:\; \|x_t-x_{t+1}\|>\epsilon\})$$
and in view of (41),
$$\mbox{Card} (\{t \in \{0,\dots,n\}:\; \|x_t-x_{t+1}\|>\epsilon\})\le 2M(\delta \kappa)^{-1}\le  Q.$$
Since $n$ is an arbitrary natural number we conclude that
$$\mbox{Card} (\{t \in \{0,1,\dots\}:\; \|x_t-x_{t+1}\|>\epsilon\}) \le  Q.$$
Since $\epsilon$ is any element of $(0,1)$ we obtain that
$$\lim_{t \to \infty}\|x_t-x_{t+1}\|=0.$$
Theorem 6.1 is proved.

\section{Proof  of Theorem 6.2}

In view of Theorem 6.1, the sequence $\{x_t\}_{t=0}^{\infty}$ is bounded. In view of (A1), it has a limit point $x_* \in C$
and a subsequence $\{x_{t_k}\}_{k=0}^{\infty}$ such that
$$x_*=\lim_{k \to \infty}x_{t_k}. \eqno (49)$$ In view of (A3) and equations (38), (39) and (49), extracting a subsequence and re-indexing, we may assume without loss of generality that
$$\phi(x_{t_k}) \subset \phi(x_*),\;k=1,2,\dots \eqno (50)$$
and that there exists
$$\widehat p \in \phi(x_*)$$
such that
$$x_{t_k+1}=\lambda_{t_k}T_{\widehat p}(x_{t_k})+(1-\lambda_{t_k})x_{t_k},\;k=1,2,\dots \eqno (51)$$
and that there exists
$$\lambda_*=\lim_{k \to \infty}\lambda_{t_k} \in [\kappa,1-\kappa]. \eqno (52)$$
It follows from Theorem 6.1, the continuity of $T_{\widehat p}$ and equations (49), (51) and (52) that
$$\lambda_*T_{\widehat p}(x_*)+(1-\lambda_*)x_*$$
$$=\lim_{k \to \infty}(\lambda_{t_k}T_{\widehat p}(x_{t_k})+(1-\lambda_{t_k})x_{t_k})$$
$$=\lim_{k \to \infty}x_{t_k+1}=\lim_{k \to \infty} x_{t_k}=x_*. \eqno (53)$$
Set
$$I_1=\{i \in \phi(x_*):\; T_i(x_*)=x_*\},\; I_2=\phi(x_*)\setminus I_1. \eqno (54)$$
In view of (53) and (54),
$$\widehat p \in I_1.$$
Fix $\delta_0\in (0,1)$ such that
$$\|x_*-T_i(x_*)\|>2\delta_0,\; i \in I_2. \eqno (55)$$
Assumption (A3), the continuity of $T_i,\;i=1,\dots,m$ and (55) imply that there exists $\delta_1 \in (0,\delta_0)$ such that for each $x \in B(x_*,\delta_1)\cap C$,
$$\phi(x)\subset \phi(x_*), \eqno (56)$$
$$\|x-T_i(x)\|>\delta_0,\; i \in I_2. \eqno (57)$$
Theorem 6.1 implies that there  exists an integer $q_1\ge 1$ such that for each integer $t \ge q_1$,
$$\|x_t-x_{t+1}\| \le \kappa\delta_0/2. \eqno (58)$$
Assume that
$$\epsilon \in (0,\delta_1), \eqno (59)$$
$$t \ge q_1 \eqno (60)$$
is an integer and that
$$\|x_t-x_*\| \le \epsilon. \eqno (61)$$
It follows from (56), (57), (59) and (61) that
$$\phi(x_t) \subset \phi(x_*) \eqno (62)$$
and
$$\|x_t-T_i(x_t)\|>\delta_0,\; i \in I_2. \eqno (63)$$
In view of (39), there exists
$$s \in \phi(x_t) \subset \phi(x_*)$$ such that
$$x_{t+1}=\lambda_tT_s(x_t)+(1-\lambda_t)x_t. \eqno (64)$$. By (38), (58) and (64),
$$\kappa \delta_0/2 \ge \|x_{t+1}-x_t\|=\lambda_t\|T_s(x_t)-x_t\|$$
and
$$\|x_t-T_s(x_t)\|\le \kappa \delta_0(2\lambda_t)^{-1}\le \delta_0/2. \eqno (65)$$
It follows from (54), (56), (57), (59), (61)  and (65) that
$$s \in I_1,\; T_s(x_*)=x_*.$$
Combined with assumption (A2) and equations (39), (61) and (64) this implies that
$$\|x_{t+1}-x_*\| =\|\lambda_t T_s(x_t)+(1-\lambda_t)x_t-x_*\|$$
$$\le \lambda_t \|T_s(x_t)-x_*\|+(1-\lambda_t)\|x_t-x_*\|$$
$$\le \|x_t-x_*\|\le \epsilon.$$
Thus we have shown that if $t \ge q_1$ is an integer and (61) holds, then $\|x_{t+1}-x_*\|\le \epsilon$.
By induction and (49), we obtain that
$$\|x_i-x_*\| \le \epsilon $$
for all sufficiently large natural numbers $i$. Since
 $\epsilon $ i an arbitrary element of $(0,\delta_1)$ we conclude that
$$\lim_{t \to \infty}x_t=x_*$$ and Theorem 6.2 is proved.

\end{document}